\documentclass[12pt]{article}

\usepackage{amsmath, epsfig, cite, setspace}
\usepackage{amssymb}
\usepackage{amsfonts}
\usepackage{latexsym}
\usepackage{amsthm}

\makeatletter
\renewcommand{\@seccntformat}[1]{{\csname the#1\endcsname}.\hspace{.5em}}
\makeatother

\newtheorem{thm}{Theorem}[section]

\newtheorem{conj}[thm]{Conjecture}
\newtheorem{lem}[thm]{Lemma}

\renewcommand{\thefootnote}{*}

\numberwithin{equation}{section}

\begin{document}

\begin{center}
{\large\bf Some congruences related to a congruence of Van Hamme}
\end{center}

\vskip 2mm \centerline{Victor J. W. Guo$^1$ and Ji-Cai Liu$^2$\footnote{Corresponding author.}}
\begin{center}
{\footnotesize $^1$School of Mathematical Sciences, Huaiyin Normal
University, Huai'an 223300, Jiangsu, People's Republic of China\\
{\tt jwguo@hytc.edu.cn } \\[10pt]

$^2$Department of Mathematics, Wenzhou University, Wenzhou 325035, People's Republic of China\\
{\tt jcliu2016@gmail.com  } }
\end{center}


\vskip 0.7cm \noindent{\bf Abstract.} We establish some supercongruences related to a supercongruence of Van Hamme, such as
\begin{align*}
\sum_{k=0}^{(p+1)/2} (-1)^k (4k-1)\frac{(-\frac{1}{2})_k^3}{k!^3}
&\equiv p(-1)^{(p+1)/2}+p^3(2-E_{p-3})\pmod{p^{4}},\\
\sum_{k=0}^{(p+1)/2} (4k-1)^5 \frac{(-\frac{1}{2})_k^4}{k!^4} &\equiv 16p\pmod{p^{4}},
\end{align*}
where $p$ is an odd prime and $E_{p-3}$ is the $(p-3)$-th Euler number. Our proof uses some congruences of Z.-W. Sun, the Wilf--Zeilberger method, Whipple's $_7F_6$ transformation, and
the software package {\tt Sigma} developed by Schneider. We also put forward two related conjectures.



\renewcommand{\thefootnote}{**}

\section{Introduction}
In 1997, Van Hamme \cite{Hamme} conjectured that Ramanujan-type formula for $1/\pi$:
\begin{align*}
\sum_{k=0}^\infty  (-1)^k (4k+1)\frac{(\frac{1}{2})_k^3}{k!^3} =\frac{2}{\pi},
\end{align*}
due to Bauer\cite{Bauer}, possesses a nice $p$-adic analogue:
\begin{align}
\sum_{k=0}^{(p-1)/2} (-1)^k (4k+1)\frac{(\frac{1}{2})_k^3}{k!^3}\equiv p(-1)^{(p-1)/2}\pmod{p^3}. \label{eq:pram}
\end{align}
Here and throughout the paper, $p$ is an odd prime and $(a)_k=a(a+1)\cdots (a+k-1)$ is the Pochhammer symbol. The supercongruence \eqref{eq:pram} was first proved by Mortenson \cite{Mortenson4}
in 2008 and reproved by Zudilin \cite{Zudilin} in 2009.  Motivated by Zudilin's work,
in 2018, the first author \cite{Guo2018} gave a $q$-analogue of \eqref{eq:pram} as follows:
\begin{align*}
\sum_{k=0}^{(p-1)/2}(-1)^k q^{k^2}[4k+1]\frac{(q;q^2)_k^3}{(q^2;q^2)_k^3}
\equiv [p]q^{(p-1)^2/4} (-1)^{(p-1)/2}\pmod{[p]^3},
\end{align*}
where $(a;q)_n=(1-a)(1-aq)\cdots (1-aq^{n-1})$ and $[n]=1+q+\cdots+q^{n-1}$. For more supercongruences and $q$-supercongruences, we refer the reader
to \cite{Guo2017,Guo2,Guo3,Guo-ef,Guo4.5,Guo5,Guo6,GS2019,GS0,GS,GW0,GW,GuoZu,Liu,Liu2019,LP,Straub,Swisher,Tauraso2}.

On the other hand, there is a similar supercongruence as follows:
\begin{align}
\sum_{k=0}^{(p+1)/2} (-1)^k (4k-1)\frac{(-\frac{1}{2})_k^3}{k!^3}
&\equiv p(-1)^{(p+1)/2} \pmod{p^{3}},  \label{eq:guo}
\end{align}
which is a special case of \cite[Theorem 1.3]{Guo-ef} or \cite[Theorem 4.9]{GuoZu} (see also \cite[Section 5]{GS}).
Note that Sun \cite{Sun} gave the following refinement of \eqref{eq:pram} modulo $p^4$:
\begin{align}
\sum_{k=0}^{(p-1)/2} (-1)^k (4k+1)\frac{(\frac{1}{2})_k^3}{k!^3}
\equiv p(-1)^{(p-1)/2}+p^3 E_{p-3}\pmod{p^4}, \label{eq:sun}
\end{align}
where $E_n$ is the $n$-th Euler number which may be defined by
\begin{align*}
\sum_{n=0}^\infty E_n \frac{x^n}{n!}=\frac{2e^x}{e^{2x}+1} \qquad \text{for}\quad |x|<\frac{\pi}{2}.
\end{align*}
Inspired by Sun's result \eqref{eq:sun}, we shall prove the following refinement of \eqref{eq:guo} modulo $p^4$.

\begin{thm}\label{t-1}
We have
\begin{align}
\sum_{k=0}^{(p+1)/2} (-1)^k (4k-1)\frac{(-\frac{1}{2})_k^3}{k!^3}
&\equiv p(-1)^{(p+1)/2}+p^3(2-E_{p-3})\pmod{p^{4}}.\label{a-0}
\end{align}
\end{thm}

We shall also prove the following weaker supercongruence.
\begin{thm}\label{t-2}
We have
\begin{align}
\sum_{k=0}^{(p+1)/2} (-1)^k (4k-1)^3\frac{(-\frac{1}{2})_k^3}{k!^3}
&\equiv 3p(-1)^{(p-1)/2}\pmod{p^{2}}.\label{new-h-1}
\end{align}
\end{thm}

Recently, the first author and Schlosser \cite{GS0} proved the following supercongruence
\begin{align*}
\sum_{k=0}^{(p+1)/2} (4k-1)\frac{(-\frac{1}{2})_k^4}{k!^4}
&\equiv -5p^4\pmod{p^{5}}.
\end{align*}
In this paper, using the same method in \cite{Liu2019}, we shall prove the following related result.
\begin{thm}\label{t-3}
We have
\begin{align*}
\sum_{k=0}^{(p+1)/2} (4k-1)^3 \frac{(-\frac{1}{2})_k^4}{k!^4} &\equiv 0\pmod{p^{4}},  \\[5pt]
\sum_{k=0}^{(p+1)/2} (4k-1)^5 \frac{(-\frac{1}{2})_k^4}{k!^4} &\equiv 16p\pmod{p^{4}}, \\[5pt]
\sum_{k=0}^{(p+1)/2} (4k-1)^7 \frac{(-\frac{1}{2})_k^4}{k!^4} &\equiv 80p\pmod{p^{4}}.
\end{align*}
\end{thm}

\section{Proof of Theorem \ref{t-1}}
We first give the following result due to Sun \cite{Sun0}.
\begin{lem}We have
\begin{align}
\sum_{k=1}^{(p-1)/2}\frac{4^k}{(2k-1){2k\choose k}}  &\equiv E_{p-1}-1+(-1)^{(p-1)/2}\pmod{p}. \label{sun-1}
\end{align}
\end{lem}

We also need the following congruence, which was given in the proof of \cite[Theorem~1.1]{Sun} implicitly.
\begin{lem}For $1\leqslant k\leqslant (p-1)/2$, we have
\begin{align}
(-1)^{(p+1)/2+k} \frac{2(\frac{1}{2})_{(p+1)/2}^2 (\frac{1}{2})_{(p-1)/2+k}}{(1)_{(p-1)/2}^2(1)_{(p+1)/2-k}(\frac{1}{2})_k^2}
\equiv \frac{p^3 4^k}{2k(2k-1){2k\choose k}}  \pmod{p^4}.  \label{sun-3}
\end{align}
\end{lem}
\begin{proof}[Proof of Theorem {\rm\ref{t-1}}.]
For all non-negative integers $n$ and $k$, define the functions
\begin{align*}
F(n,k) &=(-1)^{n+k} \frac{(4n-1)(-\frac{1}{2})_n^2 (-\frac{1}{2})_{n+k}}{(1)_n^2(1)_{n-k}(-\frac{1}{2})_k^2}, \\[5pt]
G(n,k) &=(-1)^{n+k} \frac{2(-\frac{1}{2})_n^2 (-\frac{1}{2})_{n+k-1}}{(1)_{n-1}^2(1)_{n-k}(-\frac{1}{2})_k^2},
\end{align*}
where we assume that $1/(1)_{m}=0$ for $m=-1,-2,\ldots.$ The functions $F(n,k)$ and $G(n,k)$ form a Wilf--Zeilberger pair (WZ-pair).
Namely, they satisfy the following relation
\begin{align}
F(n,k-1)-F(n,k)=G(n+1,k)-G(n,k).  \label{eq:fnk-gnk}
\end{align}
This WZ-pair is similar to one WZ-pair in \cite{Zudilin} and can be found in the spirit of \cite{EZ,PWZ}.
Summing \eqref{eq:fnk-gnk} over $n$ from $0$ to $(p+1)/2$, we obtain
\begin{align}
\sum_{n=0}^{(p+1)/2}F(n,k-1)-\sum_{n=0}^{(p+1)/2}F(n,k)=G\left(\frac{p+3}{2},k\right)-G(0,k)=G\left(\frac{p+3}{2},k\right).  \label{eq:fnk-gn0-00}
\end{align}
Summing \eqref{eq:fnk-gn0-00} further over $k=1,2,\ldots, (p+1)/2$, we get
\begin{align}
\sum_{n=0}^{(p+1)/2}F(n,0)
=F\left(\frac{p+1}{2},\frac{p+1}{2}\right)+\sum_{k=1}^{(p+1)/2}G\left(\frac{p+3}{2},k\right), \label{eq:new-0}
\end{align}
where we have used $F(n,k)=0$ for $n<k$.

It is easy to see that
\begin{align}
F\left(\frac{p+1}{2},\frac{p+1}{2}\right)
&=\frac{(2p+1)(-\frac{1}{2})_{p+1}}{(1)_{(p+1)/2}^2}
=-\frac{2p(2p+1)}{4^p(p+1)^2}{2p\choose p}{p-1\choose (p-1)/2}  \notag \\[5pt]
&\equiv (-1)^{(p+1)/2}\frac{p(2p+1)}{(p+1)^2} \equiv (-1)^{(p-1)/2}(p^3-p)\pmod{p^4},  \label{xx}
\end{align}
where we have used Wolstenholme's congruence \cite{Wolstenholme}:
$$
{2p\choose p}\equiv 2\pmod{p^3}\quad\text{for}\ p>3,
$$
and Morley's congruence \cite{Morley}:
$$
{p-1\choose (p-1)/2}\equiv (-1)^{(p-1)/2}4^{p-1}\pmod{p^3}\quad\text{for}\ p>3.
$$
Moreover, we have
\begin{align*}
\sum_{k=1}^{(p+1)/2}G\left(\frac{p+3}{2},k\right)
&=G\left(\frac{p+3}{2},1\right)+\sum_{k=1}^{(p-1)/2}G\left(\frac{p+3}{2},k+1\right) \\[5pt]
&=(-1)^{(p-1)/2}\frac{p^3}{(p+1)^3 2^{3(p-1)}}{p-1\choose (p-1)/2}^3  \\[5pt]
&\qquad{}+\sum_{k=1}^{(p-1)/2}(-1)^{(p+5)/2+k} \frac{2(-\frac{1}{2})_{(p+3)/2}^2 (-\frac{1}{2})_{(p+3)/2+k}}{(1)_{(p+1)/2}^2(1)_{(p+1)/2-k}(-\frac{1}{2})_{k+1}^2}  \\[5pt]
&\equiv p^3- \sum_{k=1}^{(p-1)/2}(-1)^{(p+1)/2+k} \frac{4(\frac{1}{2})_{(p+1)/2}^2 (\frac{1}{2})_{(p+1)/2+k}(\frac{p}{2}+k)}
{(1)_{(p-1)/2}^2(1)_{(p+1)/2-k}(\frac{1}{2})_{k}^2 (p+1)^2} \pmod{p^4}.
\end{align*}
By \eqref{sun-3}, modulo $p^4$ we may write the right-hand side of the above congruence as
\begin{align*}
p^3- \frac{p^3}{(p+1)^2}\sum_{k=1}^{(p-1)/2}\frac{4^k}{(2k-1){2k\choose k}}.
\end{align*}
Therefore,  by \eqref{sun-1}, we obtain
\begin{align}
\sum_{k=1}^{(p+1)/2}G\left(\frac{p+3}{2},k\right)\equiv p^3-p^3(E_{p-3}-1+(-1)^{(p-1)/2})\pmod{p^4}. \label{yy}
\end{align}
Substituting \eqref{xx} and \eqref{yy} into \eqref{eq:new-0}, we arrive at \eqref{a-0}.
\end{proof}

\section{Proof of Theorem \ref{t-2}}
Recall that the {\it gamma function} $\Gamma(z)$, for any complex number $z$ with the real part positive, may be defined by
\begin{align*}
\Gamma(z)=\int_{0}^{\infty}x^{z-1}e^{-x} dx,
\end{align*}
and can be uniquely analytically extended to a meromorphic function defined for all complex numbers $z$, except for non-positive integers.
It is worthwhile to mention that the gamma function has the property $\Gamma(z+1)=z\Gamma(z)$.

We need the following hypergeometric identity, which is a specialization of Whipple's
$_7F_6$ transformation (see \cite[p. 28]{Bailey}):
\begin{align}
&{}_6F_{5}\left[\begin{array}{cccccc}
a, & 1+\frac{1}{2}a, & b,     & c,      & d,      & e\\[5pt]
   & \frac{1}{2}a,   & 1+a-b, & 1+a-c,  & 1+a-d,  & 1+a-e
\end{array};-1
\right]  \notag\\[5pt]
&\quad=\frac{\Gamma(1+a-d)\Gamma(1+a-e)}{\Gamma(1+a)\Gamma(1+a-d-e)}
{}_3F_{2}\left[\begin{array}{ccc}
1+a-b-c,       & d,      & e      \\[5pt]
               &  1+a-b, & 1+a-c
\end{array};1
\right],  \label{eq:6f5}
\end{align}
where
$$
{}_{r+1}F_{r}\left[\begin{array}{cccc}
a_0, & a_1, & \ldots,     & a_{r}\\[5pt]
     & b_1, & \ldots,     & b_{r}
\end{array};z
\right]
=\sum_{k=0}^\infty\frac{(a_0)_k (a_1)_k\cdots (a_r)_k}{k!(b_1)_k\cdots (b_r)_k}z^k.
$$

Motivated by McCarthy and Osburn \cite{MO} and Mortenson \cite{Mortenson4}, we take the following choice of variables in \eqref{eq:6f5}.
Letting $a=-\frac{1}{2}$, $b=c=\frac{3}{4}$, $d=\frac{-1-p}{2}$, and $e=\frac{-1+p}{2}$, we conclude immediately that
\begin{align*}
&{}_6F_{5}\left[\begin{array}{cccccc}
-\frac{1}{2}, & \frac{3}{4}, & \frac{3}{4},     & \frac{3}{4},      & \frac{-1-p}{2},      & \frac{-1+p}{2}\\[5pt]
   & -\frac{1}{4},   & -\frac{1}{4}, & -\frac{1}{4},  & 1+\frac{p}{2},  & 1-\frac{p}{2}
\end{array};-1
\right]\\[5pt]
&\quad =\frac{\Gamma(1+\frac{p}{2})\Gamma(1-\frac{p}{2})}{\Gamma(\frac{1}{2})\Gamma(\frac{3}{2})}
{}_3F_{2}\left[\begin{array}{ccc}
-1,       & \frac{-1-p}{2},      & \frac{-1+p}{2}      \\[5pt]
               &  -\frac{1}{4}, & -\frac{1}{4}
\end{array};1
\right].  
\end{align*}
It is easy to see that, for $k=0,1,\ldots,(p-1)/2$,
\begin{align*}
\frac{\left(\frac{-1-p}{2}\right)_k \left(\frac{-1+p}{2}\right)_k}
{\left(1+\frac{p}{2}\right)_k \left(1-\frac{p}{2}\right)_k}
=\prod_{j=1}^k \frac{(2j-3)^2-p^{2}}{4j^2-p^{2}}\equiv \frac{\left(-\frac{1}{2}\right)_k^2}{k!^2} \pmod{p^{2}}.
\end{align*}
Applying the property $\Gamma(x+1)=x\Gamma(x)$, we obtain
$$
\frac{\Gamma(1+\frac{p}{2})\Gamma(1-\frac{p}{2})}{\Gamma(\frac{1}{2})\Gamma(\frac{3}{2})}
=p(-1)^{(p-1)/2},
$$
and so
\begin{align*}
-\sum_{k=0}^{(p+1)/2} (-1)^k (4k-1)^3\frac{(-\frac{1}{2})_k^3}{k!^3}
&\equiv p (-1)^{\frac{(p-1)}{2}}
{}_3F_{2}\left[\begin{array}{ccc}
-1,       & \frac{-1-p}{2},      & \frac{-1+p}{2}      \\[5pt]
               &  -\frac{1}{4}, & -\frac{1}{4}
\end{array};1
\right]  \\
&= p (-1)^{\frac{(p-1)}{2}}(1-4(1-p^2)) \\
&\equiv -3p(-1)^{\frac{(p-1)}{2}}\pmod{p^{2}},
\end{align*}
as desired.

\section{Proof of Theorem \ref{t-3}}
Let $H_n^{(2)}=\sum_{j=1}^n\frac{1}{j^2}$.  The following lemma plays an important role in our proof.
\begin{lem}\label{lem-f}
For any integer $n\geqslant 2$ and odd positive integer $3\leqslant m\leqslant 7$, we have
\begin{align}
&\sum_{k=0}^n(4k-1)^m\frac{\left(-\frac{1}{2}\right)_k^2(-n)_k(n-1)_k}{(1)_k^2\left(n+\frac{1}{2}\right)_k\left(\frac{3}{2}-n\right)_k}
=f_m(n),\label{b-1}\\
&\sum_{k=0}^n(4k-1)^m\frac{\left(-\frac{1}{2}\right)_k^2(-n)_k(n-1)_k}{(1)_k^2\left(n+\frac{1}{2}\right)_k\left(\frac{3}{2}-n\right)_k}
\sum_{j=1}^k\left(\frac{1}{4j^2}-\frac{1}{(2j-3)^2}\right)=g_m(n),\label{b-2}
\end{align}
where $f_m(n)$ and $g_m(n)$ are listed in the following table:
\begin{table}[h]
\caption{Values of $f_m(n)$ and $g_m(n)$.}
\begin{center}
\renewcommand\arraystretch{1.8}
\begin{tabular}{|c|p{3cm}|p{11cm}|}
\hline
$m$ & $f_m(n)$&$g_m(n)$ \\ \hline
$3$& {\scriptsize$0$}& $\frac{(2 n-1)^2 \left(7 n^2-7 n+1\right)}{4n^2 (n-1)^2 }$ \\ \hline
$5$& {\scriptsize$-64n(n-1)(2n-1)$}&$\frac{(2 n-1) (256 n^6-640 n^5+320 n^4+414 n^3-493 n^2+145 n-1)}{4 n^2(n-1)^2 }$
{\scriptsize$+32 n(n-1)  (2 n-1)H_n^{(2)}$} \\ \hline
$7$& {\scriptsize$-64n(n-1)(2n-1)(24n^2-24n+11)$}&$\frac{(2 n-1) (6144 n^8-21504 n^7+24320 n^6-1920 n^5-18496 n^4+17582 n^3-7557 n^2+1433 n-1)}{4 n^2(n-1)^2 }$\\[-20pt]
&&{\scriptsize$+32 n (n-1)(2 n-1) \left(24 n^2-24 n+11\right)H_n^{(2)}$ }\\
\hline
\end{tabular}
\end{center}
\end{table}
\end{lem}

\begin{proof}
By using the package {\tt Sigma} due to Schneider \cite{schneider-slc-2007}, one
can automatically discover and prove \eqref{b-1} and \eqref{b-2}. For example,
the steps to discover and prove \eqref{b-2} for $m=5$ are as follows.

Define the following sum:\\[5pt]
{\scriptsize {\sf In[1]}:=mySum=$\sum_{k=0}^n(4k-1)^5\frac{\left(-\frac{1}{2}\right)_k^2
(-n)_k(n-1)_k}{(1)_k^2\left(n+\frac{1}{2}\right)_k\left(\frac{3}{2}-n\right)_k}\sum_{j=1}^k\left(\frac{1}{4j^2}
-\frac{1}{(2j-3)^2}\right)$}

Compute the recurrence for this sum:\\[5pt]
{\scriptsize {\sf In[2]}:=rec=GenerateRecurrence[mySum,$n$][[$1$]]}\\[5pt]
{\scriptsize
\begin{doublespace}
\noindent \text{\sf Out[2]}=\((1+n) (1+2 n) \text{SUM}[n]-(-1+n) (-1+2 n) \text{SUM}[1+n]\\[5pt]
==\frac{(-1+2 n) (1+2 n) \left(78-183 n^2+109 n^4\right)}{2 (-1+n)^2 n (1+n)^2}+\frac{n
(-1+2 n)^3 (1+2 n)^2 (3+2 n) (3+4 n)^4 \left(-\frac{1}{2}\right)_n^2 (-1+n)_n (-n)_n}{2 (1+n)^3 (1+4 n) (1)_n^2 \left(\frac{3}{2}-n\right)_n \left(\frac{1}{2}+n\right)_n}\\[5pt]
-\frac{n
(-1+2 n)^3 (1+2 n)^2 (3+2 n) (3+4 n)^4 \left(-\frac{1}{2}\right)_n^2 (-1+n)_n (-n)_n}{2 (1+n)^3 (1+4 n) (1)_n^2 \left(\frac{1}{2} (3-2 n)\right)_n
\left(\frac{1}{2} (1+2 n)\right)_n}\\[5pt]
+\left(-\frac{n (-1+2 n)^3 (1+2 n) (3+4 n)^4 \left(-\frac{1}{2}\right)_n^2 (-1+n)_n (-n)_n}{2 (1+n) (1+4 n) (1)_n^2
\left(\frac{3}{2}-n\right)_n \left(\frac{1}{2}+n\right)_n}+\frac{n (-1+2 n)^3 (1+2 n) (3+4 n)^4 \left(-\frac{1}{2}\right)_n^2 (-1+n)_n (-n)_n}{2
(1+n) (1+4 n) (1)_n^2 \left(\frac{1}{2} (3-2 n)\right)_n \left(\frac{1}{2} (1+2 n)\right)_n}\right) \left(\underset{\iota _1=1}{\overset{n}{\sum
^.}}\frac{1}{\iota _1^2}\right)\\[5pt]
+\left(\frac{2 n (-1+2 n)^3 (1+2 n) (3+4 n)^4 \left(-\frac{1}{2}\right)_n^2 (-1+n)_n (-n)_n}{(1+n) (1+4 n) (1)_n^2
\left(\frac{3}{2}-n\right)_n \left(\frac{1}{2}+n\right)_n}-\frac{2 n (-1+2 n)^3 (1+2 n) (3+4 n)^4 \left(-\frac{1}{2}\right)_n^2 (-1+n)_n (-n)_n}{(1+n)
(1+4 n) (1)_n^2 \left(\frac{1}{2} (3-2 n)\right)_n \left(\frac{1}{2} (1+2 n)\right)_n}\right) \left(\underset{\iota _1=1}{\overset{n}{\sum ^.}}\frac{1}{\left(-1+2
\iota _1\right){}^2}\right)\)
\end{doublespace}}

Now we solve this recurrence:\\[5pt]
{\scriptsize
{\sf In[3]}:=recSol=SolveRecurrence[rec,SUM[$n$]]}\\[5pt]
{\scriptsize
\begin{doublespace}
\noindent \text{\sf Out[3]}=\(\left\{\{0,(-1+n) n (-1+2 n)\},\right.\\[5pt]
\left.\left\{1,\frac{(-1+2 n) \left(-1+145 n-493 n^2+670 n^3-448 n^4+128 n^5\right)}{4 (-1+n)^2 n^2}+32 (-1+n)
n (-1+2 n) \left(\underset{\iota _1=1}{\overset{n}{\sum ^.}}\frac{1}{\iota _1^2}\right)\right\}\right\}\)
\end{doublespace}}

 Finally, we combine the solutions to represent mySum:\\[5pt]
{\scriptsize {\sf In[4]}:=FindLinearCombination[recSol,mySum,$1$]}\\[5pt]
{\scriptsize\begin{doublespace}
\noindent \text{\sf Out[4]}=\(\frac{(-1+2 n) \left(-1+145 n-493 n^2+414 n^3+320 n^4-640 n^5+256 n^6\right)}{4 (-1+n)^2 n^2}+32 (-1+n) n (-1+2 n) \left(\underset{\iota
_1=1}{\overset{n}{\sum ^.}}\frac{1}{\iota _1^2}\right)\)
\end{doublespace}}

Thus, we discover and prove \eqref{b-2} for $m=5$.
\end{proof}

\begin{proof}[Proof of Theorem {\rm\ref{t-3}}.]
Observe that
\begin{align*}
\frac{(2j-3)^2-x^2}{(2j)^2-x^2}=\left(\frac{2j-3}{2j}\right)^2-\frac{3(4j-3)}{(2j)^4}x^2+O(x^4),
\end{align*}
and
\begin{align*}
\prod_{j=1}^k(a_j+b_jx^2)=\left(\prod_{j=1}^ka_j\right)\left(1+x^2\sum_{j=1}^k\frac{b_j}{a_j}\right)
+O(x^4).
\end{align*}
Hence, for $0\le k\le (p+1)/2$, we have
\begin{align}
\frac{(\frac{-1-p}{2})_k(\frac{-1+p}{2})_k}{(1+\frac{p}{2})_k(1-\frac{p}{2})_k}
&=\prod_{j=1}^k\frac{(2j-3)^2-p^2}{(2j)^2-p^2}\notag\\
&\equiv\frac{(-\frac{1}{2})_k^2}{k!^2}
\left(1+p^2\sum_{j=1}^k\left(\frac{1}{(2j)^2}-\frac{1}{(2j-3)^2}\right)\right)\pmod{p^4}.\label{new-1}
\end{align}
Letting $n=(p+1)/2$ in \eqref{b-1} and using \eqref{new-1}, we get
\begin{align}
&\sum_{k=0}^{(p+1)/2} (4k-1)^m \frac{(-\frac{1}{2})_k^4}{k!^4}\left(1+p^2\sum_{j=1}^k\left(\frac{1}{(2j)^2}-\frac{1}{(2j-3)^2}\right)\right)\notag\\
&\equiv f_m\left(\frac{p+1}{2}\right)\pmod{p^4}.\label{b-3}
\end{align}

Furthermore, it follows from \eqref{new-1} that
\begin{align}
\frac{(\frac{-1-p}{2})_k(\frac{-1+p}{2})_k}{(1+\frac{p}{2})_k(1-\frac{p}{2})_k}
\equiv \frac{(-\frac{1}{2})_k^2}{k!^2}\pmod{p^2}.\label{new-2}
\end{align}
Letting $n=(p+1)/2$ in \eqref{b-2} and noticing \eqref{new-2}, we obtain
\begin{align}
\sum_{k=0}^{(p+1)/2}(4k-1)^m \frac{(-\frac{1}{2})_k^4}{k!^4}\sum_{j=1}^k\left(\frac{1}{(2j)^2}-\frac{1}{(2j-3)^2}\right)
\equiv g_m\left(\frac{p+1}{2}\right)\pmod{p^2}.\label{new-3}
\end{align}
Finally, combining \eqref{b-3} and \eqref{new-3} we are led to
\begin{align}
\sum_{k=0}^{(p+1)/2} (4k-1)^m \frac{(-\frac{1}{2})_k^4}{k!^4}
\equiv f_m\left(\frac{p+1}{2}\right)-p^2g_m\left(\frac{p+1}{2}\right)\pmod{p^4}.\label{b-8}
\end{align}
Note that
\begin{align}
H_{(p+1)/2}^{(2)}=\left(\frac{2}{p+1}\right)^2+H_{(p-1)/2}^{(2)}\equiv 4\pmod{p}.\label{b-9}
\end{align}
The proof then follows from \eqref{b-8}, \eqref{b-9} and Lemma \ref{lem-f}.
\end{proof}

\section{Two open problems}
We end the paper with the following two conjectures, which are generalizations of Theorems \ref{t-2} and \ref{t-3}.
Note that there are similar unsolved conjectures in \cite{Guo2017}.
\begin{conj}\label{conj:conj-cm}
For any odd positive integer $m$, there exists an integer $c_m$ such that, for any odd prime $p$ and positive integer $r$, there hold
\begin{align*}
\sum_{k=0}^{\frac{p^r+1}{2}} (-1)^k (4k-1)^m \frac{(-\frac{1}{2})_k^3}{k!^3}
&\equiv c_m p^r(-1)^{\frac{(p-1)r}{2}}\pmod{p^{r+2}}, \\[5pt]
\sum_{k=0}^{p^r-1} (-1)^k (4k-1)^m \frac{(-\frac{1}{2})_k^3}{k!^3}
&\equiv c_m p^r(-1)^{\frac{(p-1)r}{2}}\pmod{p^{r+2}}.
\end{align*}
In particular, we have $c_1=-1$, $c_3=3$, $c_5=23$, $c_7=-5$, $c_9=1647$, and $c_{11}=-96973$.
\end{conj}

\begin{conj}\label{conj:conj-dm}
For any odd positive integer $m$, there exists an integer $d_m$ such that, for any odd prime $p$ and positive integer $r$, there hold
\begin{align*}
\sum_{k=0}^{\frac{p^r+1}{2}} (4k-1)^m \frac{(-\frac{1}{2})_k^4}{k!^4}
&\equiv d_m p^r \pmod{p^{r+3}}, \\[5pt]
\sum_{k=0}^{p^r-1} (4k-1)^m \frac{(-\frac{1}{2})_k^4}{k!^4}
&\equiv d_m p^r \pmod{p^{r+3}}.
\end{align*}
In particular, we have $d_1=d_3=0$, $d_5=16$, $d_7=80$, $d_9=192$, $d_{11}=640$, $d_{13}=-3472$, and $d_{15}=138480$.
\end{conj}

Note that, Conjecture \ref{conj:conj-cm} is true for $m=1$ by \cite[Theorem 1.3]{Guo-ef}, and Conjecture \ref{conj:conj-dm} is also true
for $m=1$ by \cite[Theorem 1.1]{GS0}.

\vskip 3mm \noindent{\bf Acknowledgments.}
The second author was supported by the National Natural Science Foundation of China (grant 11801417).

\end{document}